\documentclass[11pt,a4paper]{amsart}

\usepackage[T1]{fontenc}
\usepackage[utf8]{inputenc}
\usepackage{lmodern}
\usepackage{microtype}
\usepackage{amsmath,amssymb,mathtools}
\usepackage{enumitem}
\usepackage[hidelinks]{hyperref}
\usepackage[nameinlink,capitalize,noabbrev]{cleveref}

\setlist[enumerate]{leftmargin=2.2em}

\newtheorem{theorem}{Theorem}[section]
\newtheorem{proposition}[theorem]{Proposition}
\newtheorem{corollary}[theorem]{Corollary}
\newtheorem{lemma}[theorem]{Lemma}
\theoremstyle{definition}
\newtheorem{definition}[theorem]{Definition}
\theoremstyle{remark}
\newtheorem{remark}[theorem]{Remark}

\crefname{theorem}{Theorem}{Theorems}
\crefname{proposition}{Proposition}{Propositions}
\crefname{corollary}{Corollary}{Corollaries}
\crefname{lemma}{Lemma}{Lemmas}
\crefname{definition}{Definition}{Definitions}
\crefname{remark}{Remark}{Remarks}

\newcommand{\End}{\operatorname{End}}
\newcommand{\Hom}{\operatorname{Hom}}
\newcommand{\Ker}{\operatorname{Ker}}
\newcommand{\CH}{\operatorname{CH}}
\newcommand{\Chow}{\operatorname{Chow}}
\newcommand{\Chowo}{\Chow^{\circ}}
\newcommand{\hcirc}[1]{h^{\circ}(#1)}
\newcommand{\QlZl}{\mathbb Q_\ell/\mathbb Z_\ell}
\newcommand{\id}{\mathrm{id}}

\title[Uniform Rost nilpotence and birational motives]{Uniform Rost nilpotence and birational motives}
\author{David Kumallagov}
\date{}


\begin{document}

\begin{abstract}
For a field extension $E/k$ and a Chow motive $M$, let $$I_E(M)=\Ker\bigl(\End_k(M)\longrightarrow\End_E(M_E)\bigr)$$
be the base-change ideal. We establish explicit nilpotence bounds for these ideals in several geometric settings.

We also prove effective generic descent: if $M$ is a summand of $h(X)(a)$ and $M_{k(X)}$
 has uniform Rost exponent $s$, then $M$ has exponent $s(\dim X+1).$ This yields the integral exponent $2n-1$ for twisted Milnor hyperplane sections of dimension $2n-2.$

In characteristic zero, a weighted local-block refinement of the multilinear Rost filtration improves the uniform surface bounds of Gille and upgrades the elementwise estimates of Rosenschon–Sawant to ideal level uniform bounds. Using the Kahn–Sujatha description of pure birational motives, we obtain quantitative lifting results from birational to ordinary Chow motives, with applications to threefolds, including those birational to toric models and those admitting a decomposition of the diagonal supported on a surface.

Finally, we extend the Kok–Zhou detector from individual correspondences to entire base-change ideals, obtaining an explicit ideal nilpotence bound from uniform annihilation of the critical refined unramified cohomology groups.

\end{abstract}

\maketitle

\section{Introduction}

Let $X$ be a smooth projective variety over a field $k$.  The Rost
nilpotence principle predicts that every correspondence on $X$ which
vanishes after a field extension is nilpotent.  The principle is known for
several important classes of varieties, but remains open in general in
dimension at least three.  Two difficulties become visible as soon as one
tries to pass from individual correspondences to an entire base-change
kernel.  First, a nil-ideal need not be nilpotent.  Second, without
additional boundedness hypotheses, ordinary Rost nilpotence is not known to
be preserved by direct sums of motives; compare the bounded exponent result
\cite[Theorem~3.7]{Gille2026}.  Gille
introduced the strong form of the principle and proved that it is additive;
he also obtained the uniform ideal bounds $9$ for arbitrary smooth
projective integral surfaces in characteristic zero and $6$ for
geometrically rational surfaces
\cite[Theorem~3.3 and Sections~4.2--4.3]{Gille2026}.

The first purpose of this paper is to make the algebra of these bounds
explicit.  For a fixed field extension $E/k$, write
\[
 I_E(M)=\Ker\bigl(\End_k(M)\longrightarrow\End_E(M_E)\bigr).
\]
If $I_E(N_i)^{s_i}=0$, then
\[
 I_E\!\left(\bigoplus_iN_i\right)^{\sum_i s_i}=0.
\]
This gives a
quantitative refinement of the direct sum theorem.  We also prove that
\[
 I_E(M)^q\subseteq\langle N\rangle_M,\quad I_E(N)^s=0
 \quad\Longrightarrow\quad
 I_E(M)^{q(2s+1)}=0,
\]
where $\langle N\rangle_M$ is the ideal generated by factorisations through
$N$.  Thus a factorisation through a sum of $m$ motives of varieties of
dimension at most two gives the exponent $q(10m+1)$; for one surface this is
$11q$, and for one geometrically rational surface it is $9q$.  At the
element level, combining Rosenschon--Sawant with the ''square trick'' gives the
sharper universal bound $11q$.

The second purpose is to make generic descent effective.  The
Vishik--Zainoulline theorem gives
\[
 I_{k(X)}(h(X))^{\dim X+1}=0.
\]
We combine this with a common-overfield argument:  if a summand $M$ of
$h(X)(a)$ has uniform exponent $s$ after passage to $k(X)$, then $M$ already
has uniform exponent $s(\dim X+1)$ over $k$.  Thus generic Artin--Tate
splitting gives the explicit exponent $\dim X+1$.  In particular, the recent
strong Rost nilpotence theorem for twisted Milnor hyperplane sections of \cite{DeClercqMarthZainoulline} admits the
integer $2n-1$ as a uniform integral exponent.

The third purpose is geometric. The Kahn--Sujatha
formula reads
\[
 \Hom_{\Chowo(k,R)}(\hcirc X,\hcirc Y)
 \simeq \CH_0(Y_{k(X)};R).
\]
A correspondence supported on $Z\times Y$, with $Z\subset X$, then factors
through positive Tate twists of smooth projective models of $Z$.  Using the support-surjectivity lemma of Diaz
\cite{Diaz2022}, we make this factorisation precise without assuming that a
cycle lifts to one arbitrarily chosen resolution.

This leads to a uniform form of strong Rost nilpotence: one exponent works
simultaneously for all field extensions.  We prove it for threefolds whose
pure birational motive is a retract of a finite sum of birational motives of
dimension at most two, and for threefolds whose generic point is supported
on a surface; equivalently, which admit the corresponding integral
decomposition of the diagonal.  The latter strengthens the ordinary Rost
nilpotence conclusion obtained by Diaz \cite{Diaz2019}.  The
blow-up formula then shows that uniform strong Rost nilpotence is a
birational invariant in dimensions at most four.

The fourth purpose is to strengthen Gille's threefold theorem
\cite{Gille2018}.  For a base-change kernel $J$ we consider its image
$J^\circ$ in the endomorphism ring of the pure birational motive.  A
weighted local-block form of Rost's filtration shows that
\[
 (J^\circ)^r=0\quad\Longrightarrow\quad J^{r+5}=0
\]
for a smooth projective geometrically integral threefold in characteristic
zero.  The Brown--Gersten--Quillen calculation of \cite{Gille2018} works for
two different correspondences and gives $(J^\circ)^2=0$ for a
three-dimensional toric model.  The condition is birationally invariant
with the same exponent, so every threefold birational to such a model has
uniform strong exponent $7$.  Gille's original conclusion was elementwise,
and his stated exponent $8$ contains an avoidable final factor in the Rost
filtration.  The same sharpening gives the surface exponents $5$ and $4$
and upgrades Gille's projective-bundle theorem from ordinary to uniform
strong Rost nilpotence in arbitrary rank.

The last section concerns refined unramified cohomology.  Kok and Zhou prove
that, for a correspondence killed by a finite Galois extension, nilpotence is
detected by a family of critical refined unramified groups, with one bound
independent of $\ell$ \cite{KokZhou}.  We pass from one correspondence to the
whole base-change kernel attached to a finite extension.  The
Hochschild--Serre filtration and the
equivariant exact sequence of Kok--Zhou give the explicit ideal bound
$J^{N+2d+1}=0$.

\medskip

\section{Chow motives and quantitative lifting}

Unless stated otherwise, $k$ is a field and $R$ is a commutative coefficient
ring with identity.  We write $\CH^i(-;R)=\CH^i(-)\otimes_{\mathbb Z}R$.
A variety is integral and of finite type over $k$.  Smooth projective
schemes with finitely many connected components are treated additively.  We
use the category $\Chow(k,R)$ of Chow motives with arbitrary Tate twists and
the covariant convention
\begin{equation}\label{eq:covariant-convention}
 \Hom\bigl(h(X)(p),h(Y)(q)\bigr)
 =\CH^{\dim Y+q-p}(X\times_kY;R)
\end{equation}
for smooth projective integral $X$ and $Y$.  In particular, the graph of a
morphism $X\to Y$ defines a morphism $h(X)\to h(Y)$.  Composition is the
usual composition of correspondences, and we write $\alpha\beta$ for
$\alpha\circ\beta$.

For every extension $E/k$ and every $M\in\Chow(k,R)$, put
\[
 I_E(M)=\Ker\bigl(\End_k(M)\longrightarrow\End_E(M_E)\bigr).
\]

\begin{definition}\label{def:nilpotence-levels}
A motive $M\in\Chow(k,R)$ satisfies \emph{Rost nilpotence} if every element
of $I_E(M)$ is nilpotent for every extension $E/k$.  It satisfies
\emph{strong Rost nilpotence} if $I_E(M)$ is a nilpotent ideal for every
$E/k$.  It satisfies \emph{uniform strong Rost nilpotence} if there is one
integer $s\geq1$ such that
\[
 I_E(M)^s=0
\]
for every field extension $E/k$.  Such an $s$ will be called a uniform Rost
exponent for $M$.
\end{definition}

The following statement will be used repeatedly.

\begin{lemma}\label{lem:direct-sum-bound}
Fix $E/k$, let $N_1,\ldots,N_m\in\Chow(k,R)$, and assume
\[
 I_E(N_i)^{s_i}=0\qquad(1\leq i\leq m),
\]
where $s_i\geq1$.
Then, for $N=\bigoplus_{i=1}^mN_i$,
\[
 I_E(N)^{s_1+\cdots+s_m}=0.
\]

\end{lemma}

\begin{proof}
Direct and elementary matrix calculation
\end{proof}

\begin{corollary}
\label{cor:low-dimensional-sums}
\begin{enumerate}[label=\textup{(\alph*)}]
\item If $M$ is a direct summand of $N$ and $I_E(N)^s=0$ for a fixed
      extension $E/k$, then $I_E(M)^s=0$.  The same assertion holds for
      uniform exponents.
\item Assume $\operatorname{char}k=0$ and use integral coefficients.  If
      \[
       N=\bigoplus_{i=1}^m h(T_i)(a_i),
       \quad \dim T_i\leq2,
      \]
      where every $T_i$ is smooth, projective, and integral, then $5m$ is a
      uniform Rost exponent for $N$.
\item Over an arbitrary field, still with integral coefficients, if the
      $T_i$ are smooth projective geometrically rational varieties of
      dimension at most two, then $4m$ is a uniform exponent.
\end{enumerate}
\end{corollary}

\begin{proof}
Part (a) follows by identifying the endomorphism ring of a retract with a
corner algebra.  The improved surface bounds $5$ and $4$ are proved in Corollary \ref{cor:improved-surfaces} below.  If $\dim T_i<2$, then $h(T_i)$ is a
direct summand of $h(T_i\times\mathbb P^{2-\dim T_i})$ by the projective
bundle formula; in the geometrically rational case the product is again
geometrically rational.  Tate twists do not change endomorphism rings.  Now
apply Lemma \ref{lem:direct-sum-bound}.
\end{proof}

\begin{definition}\label{def:factorisation-ideal}
For motives $M,N\in\Chow(k,R)$, let
\[
 \langle N\rangle_M=
 \left\{\sum_{j=1}^r v_ju_j:\
 u_j:M\to N,\ v_j:N\to M,\ r<\infty\right\}.
\]
This is the two-sided ideal of $\End(M)$ generated by morphisms which factor
through $N$.
\end{definition}

\begin{lemma}\label{lem:element-lifting}
Let $a\in I_E(M)$ and suppose that
\[
 a^q=vu,\qquad u:M\to N,\quad v:N\to M.
\]
Here $q\geq1$.
Put $K=I_E(N)$.
\begin{enumerate}[label=\textup{(\alph*)}]
\item If every $x\in K$ satisfies $x^r=0$ for a fixed $r\geq1$, then
      $a^{q(2r+1)}=0$.
\item If $K^s=0$ for an $s\geq1$, then $a^{q(2s+1)}=0$.
\end{enumerate}
\end{lemma}

\begin{proof}
Set $b=uv\in\End(N)$.  Since $a_E=0$,
\[
 (b_E)^2=u_Ev_Eu_Ev_E=u_E(a_E^q)v_E=0,
\]
and hence $b^2\in K$.  Thus $b^{2r}=0$ in case (a), and $b^{2s}=0$ in
case (b).  Finally,
\[
 a^{q(2t+1)}=(vu)^{2t+1}=v(uv)^{2t}u=vb^{2t}u,
\]
with $t=r$ or $t=s$.
\end{proof}

The factorisation argument in Lemma \ref{lem:element-lifting} is the square trick
used by Diaz \cite[Claim~0.6]{Diaz2022}.  

\begin{theorem}\label{thm:ideal-lifting}
Fix $E/k$ and put $J=I_E(M)$ and $K=I_E(N)$.  If
\[
 K^s=0,\quad J^q\subseteq\langle N\rangle_M,
\]
for integers $q,s\geq1$, then
\[
 J^{q(2s+1)}=0.
\]
\end{theorem}

\begin{proof}
Take $2s+1$ elements $A_0,\ldots,A_{2s}$ of $J^q$.  Expand only the even
factors $A_{2i}$ as finite sums of pure factorisations $v_i u_i$ through
$N$.  Each multilinear summand has the form
\[
 v_0(u_0A_1v_1)(u_1A_3v_2)\cdots
 (u_{s-1}A_{2s-1}v_s)u_s.
\]
The $s$ bracketed endomorphisms of $N$ belong to $K$, because every odd
factor becomes zero over $E$.  Their product vanishes, hence
$(J^q)^{2s+1}=0$.
\end{proof}

\begin{corollary}\label{cor:surface-bounds}
Assume $\operatorname{char}k=0$, use integral coefficients, fix $E/k$, put
$J=I_E(M)$, and let $q\geq1$.
\begin{enumerate}[label=\textup{(\alph*)}]
\item If $a\in J$ and $a^q$ factors through $h(S)(t)$ for one smooth
      projective integral $S$ of dimension at most two, then
      \[
       a^{11q}=0.
      \]
      If $E/k$ is finite Galois, one has the refined bound
      \[
       a^{q(2\min\{\operatorname{cd}(k),2\dim S\}+3)}=0,
      \]
      where $\operatorname{cd}(k)=\sup_\ell\operatorname{cd}_\ell(k)$.
\item If
      \[
       J^q\subseteq\left\langle
       \bigoplus_{i=1}^m h(S_i)(t_i)\right\rangle_M,
       \quad \dim S_i\leq2,
      \]
      then $J^{q(10m+1)}=0$.
\item If all the $S_i$ in (b) are geometrically rational surfaces, then
      $J^{q(8m+1)}=0$.
\end{enumerate}
In particular, a fixed single surface factorisation gives $J^{11q}=0$, and
a fixed geometrically rational surface gives $J^{9q}=0$.
\end{corollary}

\begin{proof}
The argument of Rosenschon--Sawant gives nilpotence exponent at most $5$ for
an individual correspondence on a smooth projective scheme of dimension at
most two which vanishes after an arbitrary base change.  For a finite Galois
extension, the same argument gives the refined exponent
$\min\{\operatorname{cd}(k),2\dim S\}+1$
\cite[pp.~427--429]{RosenschonSawant}.
Now apply Lemma \ref{lem:element-lifting}.  Parts (b) and (c) follow from Corollary
\ref{cor:low-dimensional-sums} and Theorem \ref{thm:ideal-lifting}.
\end{proof}

\section{Effective generic descent}

We first record the known estimate in the form needed below.

\begin{theorem}\label{thm:generic-kernel}
Let $X$ be a smooth projective integral variety of dimension $d$ over $k$.
For every coefficient ring $R$,
\[
 I_{k(X)}(h(X))^{d+1}=0.
\]
The same estimate holds for every direct summand of every Tate twist of
$h(X)$.
\end{theorem}

\begin{proof}
The first assertion is the ideal form of the motivic splitting lemma of
Vishik--Zainoulline, as made explicit in
\cite[Theorem~4.1]{Gille2026}; see also
\cite[Lemma~3.2]{VishikZainoulline}.  If $M=(h(X)(a),e)$ is a direct
summand, then $I_{k(X)}(M)$ is the corresponding corner of
$I_{k(X)}(h(X)(a))$.  A product of $d+1$ elements of this corner belongs to
the $(d+1)$-st power of the latter ideal, and Tate twists do not change
endomorphism rings.
\end{proof}

Recall that an \emph{Artin--Tate motive} is a finite direct sum
\[
 A=\bigoplus_{i=1}^r A_i(a_i),
\]
where, after grouping equal twists, the $a_i$ are pairwise distinct and each
$A_i$ is a direct summand of the motive of a smooth projective
zero-dimensional scheme.

\begin{lemma}\label{lem:artin-tate-faithful}
For every field extension $E/k$ and every coefficient ring $R$, base change
\[
 \End_k(A)\longrightarrow\End_E(A_E)
\]
is injective for an Artin--Tate motive $A$.  Thus $1$ is a uniform Rost
exponent for $A$.
\end{lemma}

\begin{proof}
If $Z$ is smooth projective and zero-dimensional, then
$\CH^0(Z\times Z;R)$ is free over $R$ on the connected components of
$Z\times Z$.  Under base change each such basis vector becomes the sum of a
nonempty set of basis vectors, and the sets arising from distinct components
are disjoint.  The resulting homomorphism is split injective as an
$R$-module map.  Passing to a direct summand amounts to taking a corner and
preserves injectivity.  Finally,
\[
 \Hom(A_i(a_i),A_j(a_j))=0\qquad(a_i\ne a_j),
\]
because a zero-dimensional scheme has no Chow groups in nonzero
codimension.  Hence the endomorphism ring of $A$ is the product of the
endomorphism rings of its equal-twist Artin blocks, and the assertion
follows.
\end{proof}

The following quantitative descent statement is the main result of this
section.

\begin{theorem}\label{thm:generic-descent}
Let $X$ be a smooth projective integral variety of dimension $d$, and let
$M$ be a direct summand of $h(X)(a)$.  Suppose that $M_{k(X)}$ has uniform
Rost exponent $s$ as a motive over $k(X)$.  Then
$s(d+1)$
is a uniform Rost exponent for $M$ over $k$.
\end{theorem}

\begin{proof}
Fix an extension $E/k$ and put $J=I_E(M)$ and
$K=I_{k(X)}(M)$.  Choose a prime ideal of the nonzero ring
$E\otimes_k k(X)$ and let $\Omega$ be the fraction field of the resulting
domain.  Then $\Omega$ is a common overfield of $E$ and $k(X)$.

For $\alpha_1,\ldots,\alpha_s\in J$, the restrictions of the $\alpha_i$ to
$M_{k(X)}$ vanish after extension to $\Omega$.  The assumed uniform
exponent over $k(X)$ therefore gives
\[
 (\alpha_1\cdots\alpha_s)_{k(X)}=0.
\]
Consequently $J^s\subseteq K$.  By \cref{thm:generic-kernel},
$K^{d+1}=0$, and hence
\[
 J^{s(d+1)}=(J^s)^{d+1}\subseteq K^{d+1}=0.
\]
The bound is independent of $E$.
\end{proof}

The ''qualitative'' common overfield argument is
\cite[Lemma~3]{DeClercqMarthZainoulline}.  The content of
\cref{thm:generic-descent} is the general exponent estimate.

\begin{corollary}\label{cor:generic-artin-tate}
In the situation of \cref{thm:generic-descent}, if $M_{k(X)}$ is
Artin--Tate (in particular, if it is split), then $d+1$ is a uniform Rost
exponent for $M$.  In particular, a generically split smooth projective
integral $d$-fold has uniform Rost exponent at most $d+1$.
\end{corollary}

\begin{proof}
Apply Lemma \ref{lem:artin-tate-faithful} and Theorem \ref{thm:generic-descent} with $s=1$.
\end{proof}

We now apply effective descent to a recent class for which strong
Rost nilpotence is known.  Let $A$ be a central simple algebra of degree
$n+1$, let $\varphi\in A$ generate a maximal commutative étale
$k$-subalgebra $L=k[\varphi]$, and consider the twisted partial flag variety
\[
 \mathcal E_A=\{I_1\subset I_n\subset A:\operatorname{rdim}I_i=i\}.
\]
The twisted Milnor hyperplane section $Y_A\subset\mathcal E_A$ is defined by
$\varphi I_1\subset I_n$.

\begin{corollary}
\label{cor:twisted-milnor}
Let $n\geq2$.  Assume that $\varphi\in A^\times$ has $n+1$ distinct
geometric eigenvalues and that $L=k[\varphi]$ is a maximal commutative
subalgebra of $A$.  Assume moreover that, for some $0\leq m\leq n$, there
is an identification
\[
 L\simeq L'\times k^{\,n-m},
 \quad [L':k]=m+1,
\]
under which
\[
 \varphi=(\varphi',a_{m+1},\ldots,a_n),\quad
 L'=k[\varphi'],\quad a_j\in k^\times,
\]
and $L'/k$ is Galois.  Assume also that $Y_A$ is smooth,
projective, and integral.  Then $2n-1$ is a uniform Rost
exponent for $h(Y_A)$ with integral coefficients.
More generally, every direct summand of
$\bigoplus_{j=1}^r h(Y_{A_j})(a_j)$ has uniform exponent at most
$\sum_{j=1}^r(2n_j-1)$.
\end{corollary}

\begin{proof}
The variety $\mathcal E_A$ has dimension $2n-1$, so its smooth hyperplane
section $Y_A$ has dimension $2n-2$.  The proof of
\cite[Theorem~4]{DeClercqMarthZainoulline} shows that $A$ splits over
$k(Y_A)$ and that $(Y_A)_{k(Y_A)}$ becomes the untwisted section $Y$.
Equations~(1)--(2) and Lemma~1 of that paper express $h(Y)$ as an
Artin--Tate motive, recursively in the mixed case $m<n$; the initial
decomposition is due to \cite[Theorem~1.1]{Marth}.  More explicitly, in the
pure Galois case $m=n$, after putting $F=k(Y_A)$ one has
\[
 h(Y_A)_F\simeq h(Y)_F\simeq
 \bigoplus_{i=0}^{n-2}h(\mathbb P^n_F)(i)
 \oplus h\bigl(\operatorname{Spec}(L\otimes_kF)\bigr)(n-1).
\]
Thus Corollary \ref{cor:generic-artin-tate} shows that $2n-1$ is a uniform exponent.  The final assertion follows from Lemma
\ref{lem:direct-sum-bound} and passage to a direct summand.
\end{proof}

\section{Support and birational retracts}

Except for assertions explicitly stated over fields of arbitrary
characteristic, we now work in characteristic zero.

\begin{lemma}\label{lem:support-factorisation}
Let $X,Y$ be smooth projective integral varieties, set $d=\dim X$, and let
$\gamma\in\CH^{\dim Y}(X\times Y;R)$.
\begin{enumerate}[label=\textup{(\alph*)}]
\item If $\gamma$ is supported on $Z\times Y$ for a proper closed subset
      $Z\subset X$, then $\gamma$ factors through
      \[
       \bigoplus_j h(T_j)(d-\dim T_j),
      \]
      where the $T_j$ are smooth projective integral varieties with
      $\dim T_j\leq\dim Z$.
\item If $\gamma$ is supported on $X\times Z$ for a closed subset
      $Z\subset Y$, then it factors through $\bigoplus_jh(T_j)$, with
      smooth projective integral $T_j$ satisfying $\dim T_j\leq\dim Z$.
\end{enumerate}
\end{lemma}

\begin{proof}
By localisation, a class supported on $Z\times Y$ is the push-forward of a
class on $Z\times Y$.  Diaz's support-surjectivity lemma
\cite[Lemma~0.5]{Diaz2022} gives a finite disjoint union
$T=\coprod_jT_j$ of smooth projective integral varieties and a proper map
$\pi:T\to Z$ such that
\[
 (\pi\times\id_Y)_*:\CH_*(T\times Y;R)
 \longrightarrow\CH_*(Z\times Y;R)
\]
is surjective.  Put $g_j:T_j\to Z\hookrightarrow X$ and lift the class to
classes $\beta_j$ on $T_j\times Y$.  If $t_j=\dim T_j$, then the transpose
of the graph of $g_j$ and the lifted class have the types
\[
 {}^t\Gamma_{g_j}:h(X)\longrightarrow h(T_j)(d-t_j),
 \quad
 \beta_j:h(T_j)(d-t_j)\longrightarrow h(Y).
\]
The standard composition formula for correspondences shows that
$\gamma=\sum_j\beta_j\,{}^t\Gamma_{g_j}$; see also
\cite[Chapter~16]{Fulton}.

For support on $X\times Z$, apply the same surjectivity statement to
$(\id_X\times\pi)_*$.  A lifted class defines a map $h(X)\to h(T_j)$, and
the graph of $T_j\to Z\hookrightarrow Y$ defines $h(T_j)\to h(Y)$.  Their
composite gives the required summand.  Tensoring the integral statement with
$R$ preserves surjectivity.
\end{proof}

Let $\Chowo(k,R)$ be the idempotent completion of the quotient of effective
Chow motives by the ideal of morphisms factoring through an object of the
form $P(1)$, see \cite{KahnSujatha}.  We write $\hcirc X$ for the image of $h(X)$.

\begin{theorem}\label{thm:KS}
For smooth projective integral $X,Y$ over a field of characteristic zero,
\[
 \Hom_{\Chowo(k,R)}(\hcirc X,\hcirc Y)
 \simeq \CH_0(Y_{k(X)};R).
\]
The quotient map is restriction to the generic point of $X$, and its kernel
consists of classes supported on $D\times Y$ for some proper closed subset
$D\subset X$.
\end{theorem}

\begin{proof}
This is \cite[Lemma~2.3.7, Theorem~2.4.2 and
Corollary~2.4.3]{KahnSujatha}.  
\end{proof}

\subsection{Birational kernel ideals and a sharp Rost filtration}

Throughout this subsection all Chow groups and motives have integral
coefficients.

For a smooth projective integral variety $X$ and an extension $E/k$, put
\[
 J_E(X)=I_E(h(X)).
\]
When $\operatorname{char}k=0$, also put
\[
 B_E(X)=\operatorname{im}\bigl(J_E(X)\longrightarrow
          \End(\hcirc X)\bigr).
\]
The second ideal retains exactly the zero-cycle information needed in
Gille's condition (RNZC); see \cite{Gille2018}.

\begin{lemma}
\label{lem:birational-ideal-CH0}
Let $r\geq1$.  The following conditions are equivalent:
\begin{enumerate}[label=\textup{(\roman*)}]
\item $B_E(X)^r=0$;
\item for every field extension $F/k$, every product of $r$ elements of
      $J_E(X)$ acts trivially on $\CH_0(X_F)$;
\item every product of $r$ elements of $J_E(X)$ acts trivially on
      $\CH_0(X_{k(X)})$.
\end{enumerate}
\end{lemma}

\begin{proof}
Let $\gamma$ be a product of $r$ elements of $J_E(X)$.  If
$\gamma^\circ=0$, then \cref{thm:KS} says that $\gamma$ is supported on
$D\times X$ for a proper closed $D\subset X$.  By Lemma
\ref{lem:support-factorisation}, it factors through a sum of motives
\[
 h(T)(\dim X-\dim T),\quad \dim T<\dim X.
\]
After any extension $F/k$, its action on
$\CH_0(X_F)=\Hom\bigl(h(\operatorname{Spec}F),h(X_F)\bigr)$
therefore factors through
\[
 \Hom\bigl(h(\operatorname{Spec}F),
            h(T_F)(\dim X-\dim T)\bigr)
 =\CH^{\dim X}(T_F)=0.
\]
This proves (i)$\Rightarrow$(ii), and (ii)$\Rightarrow$(iii) is immediate.
Conversely, Theorem \ref{thm:KS} identifies $\gamma^\circ$ with
$\gamma_*[\eta_X]$.  Condition~(iii) kills this class, hence
$\gamma^\circ=0$ and (i) follows.
\end{proof}

\begin{proposition}
\label{prop:birational-ideal-RNZC}
If $X$ and $Y$ are smooth projective geometrically integral birational
varieties in characteristic zero, then $B_E(X)^r=0$ for every $E/k$ if and
only if $B_E(Y)^r=0$ for every $E/k$.  
\end{proposition}

\begin{proof}
Put $L=k(X)=k(Y)$.  Let $a:h(X)\to h(Y)$ and $b:h(Y)\to h(X)$ be the
closures of the graphs of inverse birational maps.  Their actions on
zero-cycles over $L$ are mutually inverse
\cite[Theorem~1.2]{Gille2018}.  For
$\gamma_1,\ldots,\gamma_r\in J_E(X)$, put
$\delta_i=a\gamma_i b\in J_E(Y)$.  On $\CH_0(Y_L)$ one has
\[
 (\delta_1\cdots\delta_r)_*
 =a_*(\gamma_1\cdots\gamma_r)_*b_*,
\]
because $(ba)_*=\id$ on $\CH_0(X_L)$.  If the left-hand side is zero, then
pre- and post-composition with $b_*$ and $a_*$ shows that the middle action
is zero.  Now apply Lemma \ref{lem:birational-ideal-CH0}; the converse is symmetric.
\end{proof}

We shall use the following weighted local-block version of Rost's
multilinear filtration.  Here $X^{(i)}$ denotes the set of points of
codimension $i$.

\begin{lemma}
\label{lem:sharp-Rost-filtration}
Over an arbitrary field, let $X$ be smooth, projective, and integral of
dimension $d$,
and let $J\subseteq\End(h(X))$ be a two-sided ideal.  Suppose that integers
$n_i\geq1$ satisfy
\[
 \bigl(J_{k(x)}^{n_i}\bigr)_*\CH_i(X_{k(x)})=0
 \qquad(x\in X^{(i)},\ 0\leq i\leq d).
\]
Here $J_{k(x)}$ denotes the image of $J$ under base change.
Then
\[
 J^{\,n_0+\cdots+n_d}=0.
\]
\end{lemma}

\begin{proof}
Let $F_p\End(h(X))$ be the support filtration generated by correspondences
whose projection to the first factor has image of dimension at most $p$;
thus $F_d=\End(h(X))$ and $F_{-1}=0$.  Write
$\beta_*\gamma=\beta\gamma$ for left composition.  We first spell out the
indexing in Rost's argument.  If $\beta$ acts trivially on
$\CH_i(X_{k(x)})$ for every $x\in X^{(i)}$, then
\begin{equation}\label{eq:sharp-Rost-step}
 \beta_*F_{d-i}\subseteq F_{d-i-1}.
\end{equation}
Indeed, for a generator whose first projection has dimension $d-i$, its
restriction over the generic point $x$ of that projection is a class in
$\CH_i(X_{k(x)})$.  The restriction of its composite with $\beta$ is zero,
so localisation moves that composite over a proper closed subset of the
projection.  Generators already in $F_{d-i-1}$ stay there; compare \cite[Lemma~2.2]{Gille2018}.

Take an arbitrary word $w$ of length $N=\sum_i n_i$ in $J$.  In our
composition convention write it as $w=\beta_d\cdots\beta_1\beta_0,$
where the rightmost block $\beta_i$ has length $n_i$ after the blocks
$\beta_0,\ldots,\beta_{i-1}$ have been removed.  Each $\beta_i$ satisfies
the local vanishing in degree $i$.  Starting with the diagonal
$\Delta_X\in F_d$, \eqref{eq:sharp-Rost-step} gives
\[
 F_d\xrightarrow{\,\beta_0\,}F_{d-1}
 \xrightarrow{\,\beta_1\,}\cdots
 \xrightarrow{\,\beta_d\,}F_{-1}=0.
\]
The resulting cycle is $w_*\Delta_X$, namely the correspondence $w$ itself.
It is therefore zero.
\end{proof}

\begin{remark}
The proof of \cite[Lemma~2.2]{Gille2018} instead evaluates the same zero
operator on one further kernel element and consequently records the weaker
power $N+1$.  Evaluation on $\Delta_X$ gives the asserted sharp power $N$;
compare the explicitly multilinear formulation in
\cite[Lemma~3.2]{Gille2026}.    
\end{remark}

\begin{corollary}
\label{cor:improved-surfaces}
Use integral coefficients.
\begin{enumerate}[label=\textup{(\alph*)}]
\item If $S$ is a smooth projective integral surface over a
      field of characteristic zero, then
      \[
       B_E(S)^3=0,\quad I_E(h(S))^5=0
      \]
      for every extension $E/k$.
\item If $S$ is geometrically rational, over a field of arbitrary
      characteristic, then $I_E(h(S))^4=0$ for every extension $E/k$.
      Its ideal zero-cycle exponent is at most $2$.
\end{enumerate}
\end{corollary}

\begin{proof}
For (a), the multilinear form of
\cite[Theorem~4]{Gille2014}, used explicitly with three different
correspondences in \cite[Section~4.3]{Gille2026}, gives
\[
 (\alpha_1\alpha_2\alpha_3)_{k(s),*}\CH_0(S_{k(s)})=0
\]
for all $s\in S$, hence ideal zero-cycle exponent $3$.  For (b), Gille's
argument gives
$(\alpha_1\alpha_2)_*\CH_0(S_{k(s)})=0$, hence exponent $2$
for the geometric kernel \cite[Section~4.2]{Gille2026}.  The same bound
holds for $J_E(S)$: over a common overfield of $E$ and $\bar k$, vanishing
over $E$ and injectivity of base change for the split motive $h(S_{\bar k})$
imply vanishing over $\bar k$.  Thus $J_E(S)$ is contained in the geometric
kernel.  In both cases one kernel element acts
trivially on divisors after every residue-field extension, by the
Hochschild--Serre argument of \cite[Lemma~1.4]{Gille2018}, and it acts
trivially on top-dimensional cycles.  Thus the local block lengths are
$(3,1,1)$ in (a) and $(2,1,1)$ in (b).  Now apply Lemma 
\ref{lem:sharp-Rost-filtration}.
\end{proof}

\begin{theorem}
\label{thm:ideal-RNZC-threefold}
Let $X$ be a smooth projective geometrically integral threefold over a field
of characteristic zero, use integral coefficients, and let $E/k$ be any
extension.  If
\[
 B_E(X)^r=0,
\]
then
\[
 I_E(h(X))^{r+5}=0.
\]
Consequently, a bound $r$ independent of $E$ gives the uniform Rost exponent
$r+5$ for $h(X)$.
\end{theorem}

\begin{proof}
Put $J=J_E(X)$.  By Lemma \ref{lem:birational-ideal-CH0}, the codimension-zero
block length is $n_0=r$.  The divisor and top-cycle arguments in the
preceding proof give $n_2=n_3=1$.

It remains to record the mixed codimension-one bound $n_1=3$.  Let
$x\in X^{(1)}$, let $z\in\CH_1(X_{k(x)})$, and choose
$\alpha_1,\alpha_2,\alpha_3\in J$.  Put $F=k(x)$.  As in
\cite[proof of Theorem~2.1]{Gille2018}, the finitely many rational
equivalences descend to a finitely generated subextension $E_0/k$ of $E$,
and there are extensions
\[
 K\supseteq E_0F\supseteq L\supseteq F
\]
such that $L/F$ is purely transcendental, $K/L$ is finite Galois, and
all three correspondences vanish over $K$.  The point $x$ induces an
$F$-point of $X_F$ and hence an $L$-point of $X_L$.  The cycle
$\xi=(\alpha_3)_*z$ also vanishes over $K$.  Apply the multilinear last
assertion of \cite[Theorem~4]{Gille2014} over $L$, with
$X_1=\operatorname{Spec}L$ and $X_2=X_3=X_L$, to $\xi_L$ and
$\alpha_{1,L},\alpha_{2,L}$.  It gives
\[
 \bigl((\alpha_1\alpha_2\alpha_3)_*z\bigr)_L=0.
\]
Since $L/F$ is purely transcendental, base change
$\CH_1(X_F)\to\CH_1(X_L)$ is injective (indeed, an isomorphism), and hence
$(\alpha_1\alpha_2\alpha_3)_*z=0$ over $F$.
Therefore the local lengths are
$(n_0,n_1,n_2,n_3)=(r,3,1,1).$
Their sum is $r+5$, and Lemma \ref{lem:sharp-Rost-filtration} applies.
\end{proof}

\begin{corollary}
\label{cor:element-RNZC}
In the situation of \cref{thm:ideal-RNZC-threefold}, let
$a\in I_E(h(X))$.  If $(a^\circ)^q=0$, then
\[
 a^{q+5}=0.
\]
If $E/k$ is finite Galois, one may take the smaller of this exponent and
$q(2\min\{\operatorname{cd}(k),4\}+3)$.
\end{corollary}

\begin{proof}
Since $(a^q)^\circ=0$, the proof of Lemma 
\ref{lem:birational-ideal-CH0}, applied to the single correspondence
$\gamma=a^q$, shows that $a^q$ acts trivially on the generic zero-cycle
group.  Apply the proof of \cref{thm:ideal-RNZC-threefold} to the
single element $a$, with local lengths $(q,3,1,1)$.  Moreover, Lemma 
\ref{lem:support-factorisation} factors $a^q$ through a finite sum of
motives $h(T)(3-\dim T)$ with $\dim T\leq2$.  Replacing $T$ by
$T\times\mathbb P^{2-\dim T}$ and using the projective-bundle formula embeds
each such motive as a summand of $h(S)(1)$ for one possibly disconnected
smooth projective surface $S$.  For a finite Galois extension the
Rosenschon--Sawant estimate for $S$, combined with the square trick as in Corollary 
\ref{cor:surface-bounds}, gives the second bound.
\end{proof}

Following \cite{Gille2018}, a three-dimensional toric model means a smooth
projective geometrically integral variety with an action of a $k$-torus
$T$ and a $T$-equivariant open immersion of a $T$-torsor.

\begin{theorem}
\label{thm:toric-strong}
Let $X$ be a smooth projective geometrically integral threefold over a field
of characteristic zero.  If $X$ is birational to a three-dimensional toric
model, then, with integral coefficients,
\[
 I_E(h(X))^7=0
\]
for every field extension $E/k$.
\end{theorem}

\begin{proof}
First let $X$ itself be a toric model and fix $E/k$.  If
$\alpha\in J_E(X)$, then $\alpha_{\bar k}=0$.  Indeed, take a common
overfield of $E$ and $\bar k$; the motive of $X_{\bar k}$ is split, so base
change on its endomorphism ring is injective.  The vanishing descends to a
finite Galois extension.

Put $F=k(X)$.  The variety $X_F$ has an $F$-point.  The
Brown--Gersten--Quillen argument in the final lemma of
\cite[pp.~70--71]{Gille2018} proves, for every $\alpha\in J_E(X)$,
\[
 \alpha_*A_0(X_F)=0,
\]
where $A_0$ is the kernel of the degree map.  For another
$\beta\in J_E(X)$, choose one finite Galois extension $k'/k$ over which both
correspondences vanish.  Geometric integrality makes $F/k$ regular, so
$k'F/F$ is Galois and is an admissible common splitting field.  For every $z\in\CH_0(X_F)$ the class $\beta_*z$ becomes zero over
$k'F$; hence it has degree zero and lies in $A_0(X_F)$.  Thus
$(\alpha\beta)_*\CH_0(X_F)=0$.  By Lemma 
\ref{lem:birational-ideal-CH0}, $B_E(X)^2=0$.

The same exponent $2$ holds on every birational smooth projective model by Proposition 
\ref{prop:birational-ideal-RNZC}.  Finally, apply
\cref{thm:ideal-RNZC-threefold} with $r=2$.
\end{proof}

\begin{lemma}
\label{lem:birational-direct-sum}
Let $Y_1,\ldots,Y_m$ be smooth projective integral varieties in
characteristic zero.  If $B_E(Y_i)^{r_i}=0$, then the image of
$I_E(\bigoplus_i h(Y_i))$ in
$\End(\bigoplus_i\hcirc{Y_i})$ has exponent at most $\sum_i r_i$.
\end{lemma}

\begin{proof}
Repeat the proof of Lemma  \ref{lem:direct-sum-bound} in the birational
category.  
\end{proof}

\begin{theorem}
\label{thm:birational-retract}
Let $X$ be a smooth projective geometrically integral threefold over a
field of characteristic zero, and use integral coefficients.  Suppose that
$\hcirc X$ is a retract of
\[
 N^\circ=\bigoplus_{i=1}^m\hcirc{Y_i}
\]
for smooth projective integral $Y_i$.  If integers $r_i\geq1$, independent
of $E$, satisfy $B_E(Y_i)^{r_i}=0$ for all extensions $E/k$, then
\[
 \sum_{i=1}^m r_i+5
\]
is a uniform Rost exponent for $h(X)$.
\end{theorem}

\begin{proof}
Lift the retraction to $u:h(X)\to N$ and $v:N\to h(X)$; thus
$v^\circ u^\circ=\id_{\hcirc X}$.  Fix $E/k$ and take
$\alpha_1,\ldots,\alpha_r\in J_E(X)$, where $r=\sum_i r_i$.  The
endomorphisms $u\alpha_jv$ belong to $I_E(N)$, and Lemma 
\ref{lem:birational-direct-sum} gives
\[
 0=(u\alpha_1v)^\circ\cdots(u\alpha_rv)^\circ
   =u^\circ(\alpha_1\cdots\alpha_r)^\circ v^\circ.
\]
Multiplying on the left by $v^\circ$ and on the right by $u^\circ$ yields
$(\alpha_1\cdots\alpha_r)^\circ=0$.  Hence $B_E(X)^r=0$. 
Applying \cref{thm:ideal-RNZC-threefold} yields the final formula.
\end{proof}

\begin{corollary}
\label{cor:threefold-birational-retract}
Let $X$ be a smooth projective geometrically integral threefold over a field of
characteristic zero.  Suppose that, in $\Chowo(k,\mathbb Z)$, $\hcirc X$ is
a retract of $\bigoplus_{i=1}^m\hcirc{Y_i}$, where $h(Y_i)$ has uniform Rost
exponent $\rho_i$.  Then
\[
 \sum_{i=1}^m\rho_i+5
\]
is a uniform Rost exponent for $h(X)$.  In particular, if every
$\dim Y_i\leq2$, then $3m+5$ is a uniform exponent.  If all $Y_i$ are
geometrically rational surfaces, the bound is $2m+5$.  If the retract has
one term $Y_A$ as in Corollary \ref{cor:twisted-milnor}, the bound is $2n+4$.
\end{corollary}

\begin{proof}
The image ideal $B_E(Y_i)$ has exponent at most $\rho_i$.  For a surface in
characteristic zero it has exponent $3$, and for a geometrically rational
surface exponent $2$, by Corollary \ref{cor:improved-surfaces}.  Curves and points
have zero birational kernel, so the displayed coarse bounds still apply.
Now use \cref{thm:birational-retract}; the twisted-Milnor assertion uses Corollary 
\ref{cor:twisted-milnor}.
\end{proof}

\begin{corollary}
\label{cor:pure-birational-threefolds}
Let $X$ and $Y$ be smooth projective geometrically integral threefolds over a field of
characteristic zero.  If $\hcirc X\simeq\hcirc Y$ in
$\Chowo(k,\mathbb Z)$, then $h(X)$ has uniform strong Rost nilpotence if and
only if $h(Y)$ does.  Consequently, uniform strong Rost nilpotence is a
stable birational invariant of smooth projective geometrically integral
threefolds; in
particular, every stably rational such threefold has uniform Rost exponent
at most $6$.  More generally, if $\rho$ is a uniform exponent for $h(Y)$,
then $\rho+5$ is one for $h(X)$.
\end{corollary}

\begin{proof}
Apply \cref{thm:birational-retract} to the isomorphism in each direction.
For a point the birational kernel is zero, for which we use exponent $1$.
Stable birational smooth
projective varieties have isomorphic pure birational motives: birational
maps become invertible, while the projective-bundle formula shows that
$\hcirc{(X\times\mathbb P^r)}\simeq\hcirc X$; see
\cite[Proposition~2.3.8 and Corollary~2.4.3]{KahnSujatha}.
\end{proof}

We next give a weaker and more intrinsic geometric condition which produces
the required motivic factorisation directly.  For a possibly singular
projective scheme $V$, the group $\CH_0(V)$ is understood in the sense of
Fulton.

\begin{theorem}
\label{thm:surface-supported-generic-point}
Let $X$ be a smooth projective geometrically integral threefold over a field of
characteristic zero.  Suppose that there is a closed subset $V\subset X$ of
dimension at most two such that
\[
 [\eta_X]\in
 \operatorname{im}\bigl(\CH_0(V_{k(X)})\longrightarrow
 \CH_0(X_{k(X)})\bigr).
\]
Equivalently, suppose that there are cycles $\Gamma_D,\Gamma_V$ such that
\[
 \Delta_X=\Gamma_D+\Gamma_V,
\]
where $\Gamma_D$ is supported on $D\times X$ for a proper closed subset
$D\subset X$, and $\Gamma_V$ is supported on $X\times V$.  Then $h(X)$ has
uniform strong Rost nilpotence with integral coefficients.  If the support
factorisation of $\Gamma_V$ uses $r$ smooth projective integral summands of
dimension at most two, then $3r+5$ is a uniform exponent.  If all these
summands are geometrically rational, then $2r+5$ is a uniform exponent.
\end{theorem}

\begin{proof}
The generic point condition lifts $[\eta_X]$ to a zero-cycle on
$V_{k(X)}$.  Spreading this cycle over a nonempty open subset of $X$ and
taking its closure gives a correspondence $\Gamma_V$ supported on
$X\times V$ whose restriction to the generic point of the first factor is
$[\eta_X]$.  By localisation,
\[
 \Delta_X=\Gamma_D+\Gamma_V,
\]
where $\Gamma_D$ is supported on $D\times X$ for a proper closed subset
$D\subset X$.  Conversely, restricting such a decomposition to the generic
point gives the stated condition; this is the classical Bloch--Srinivas
generic point argument \cite{BlochSrinivas}.

In the pure birational category $\Gamma_D$ vanishes.  By the second part of Lemma 
\ref{lem:support-factorisation}, $\Gamma_V$ factors through
$N=\bigoplus_{j=1}^r h(T_j)$ with $\dim T_j\leq2$.  Thus $\hcirc X$ is a
retract of $\bigoplus_j\hcirc{T_j}$.  The birational kernel exponents are at
most $3$ for arbitrary surfaces in characteristic zero and $2$ for
geometrically rational surfaces, by Corollary \ref{cor:improved-surfaces}.  Finally, it remains to apply the 
\cref{thm:birational-retract}.
\end{proof}

\begin{corollary}\label{cor:universal-support}
Let $X$ be as in \cref{thm:surface-supported-generic-point}.  If there is a
closed subset $V\subset X$ of dimension at most two such that
\[
 \CH_0(V_F)\longrightarrow\CH_0(X_F)
\]
is surjective for every extension $F/k$, then $h(X)$ has uniform strong Rost
nilpotence with integral coefficients.
\end{corollary}

\begin{proof}
Take $F=k(X)$ and apply
\cref{thm:surface-supported-generic-point}.
\end{proof}

\begin{remark}
Diaz proved ordinary integral Rost nilpotence from the universal support
hypothesis in dimension three \cite[Corollary~2.10]{Diaz2019}, with the
correction in \cite{Diaz2022}.  The uniform strong conclusion in Corollary 
\ref{cor:universal-support} uses the surface ideal bounds and the explicit
direct-sum estimate; \cref{thm:surface-supported-generic-point} requires only
the single generic-point condition actually used in the argument.
\end{remark}

\section{Motivic operations and birationality}

\begin{proposition}\label{prop:projective-bundle-bound}
Let $X$ be a smooth projective $k$-scheme, let $\mathcal E$ be a vector
bundle of rank $r+1$ on $X$, and let $a$ be a uniform Rost exponent for
$h(X)$.  Then $a(r+1)$ is a uniform Rost exponent for
$h(\mathbb P_X(\mathcal E))$.  Uniform strong Rost nilpotence holds for
$\mathbb P_X(\mathcal E)$ if and only if it holds for $X$.
\end{proposition}

\begin{proof}
The motivic projective-bundle formula gives
\[
 h(\mathbb P_X(\mathcal E))\simeq\bigoplus_{i=0}^r h(X)(i).
\]
All that remains is to apply Lemma \ref{lem:direct-sum-bound}.  Conversely, $h(X)$ is a direct summand
of the displayed motive.
\end{proof}

\begin{corollary}
\label{cor:Gille-projective-bundles}
Let $S$ be a smooth projective geometrically integral surface, use integral
coefficients, and let $\mathcal E$ have rank $r+1$.
\begin{enumerate}[label=\textup{(\alph*)}]
\item If $\operatorname{char}k=0$, then $5(r+1)$ is a uniform exponent for
      $h(\mathbb P_S(\mathcal E))$.
\item If $S$ is geometrically rational, over a field of arbitrary
      characteristic, then $4(r+1)$ is a uniform exponent.
\item If $r=1$ and $\operatorname{char}k=0$, the sharper exponents are $8$
      for arbitrary $S$ and $7$ for geometrically rational $S$.
\end{enumerate}
\end{corollary}

\begin{proof}
Parts (a) and (b) follow from Corollary \ref{cor:improved-surfaces} and the
projective-bundle formula.  If $r=1$, that formula gives
$\hcirc{\mathbb P_S(\mathcal E)}\simeq\hcirc S$.  Apply
\cref{thm:birational-retract} with birational kernel exponent $3$, or $2$
in the geometrically rational case.
\end{proof}

\begin{remark}
This strengthens \cite[Theorem~3.1]{Gille2018}: the conclusion is uniform
strong rather than elementwise, the vector bundle may have arbitrary rank,
and the numerical bounds for ranks two and three improve from
$18,27$ to $8,15$ in characteristic zero and from $12,18$ to $7,12$ for
geometrically rational surfaces in characteristic zero.  In positive
characteristic part~(b) gives the uniform bounds $8,12$.
\end{remark}

\begin{proposition}\label{prop:blow-up-bound}
Let $X$ be a smooth projective $k$-scheme, let $Z\subset X$ be a smooth
closed subscheme of pure codimension $c\geq2$, and put
$\widetilde X=\operatorname{Bl}_Z X$.
\begin{enumerate}[label=\textup{(\alph*)}]
\item If $a$ and $b$ are uniform Rost exponents for $h(X)$ and $h(Z)$,
      respectively, then
      \[
       a+b(c-1)
      \]
      is a uniform Rost exponent for $h(\widetilde X)$.
\item Uniform strong Rost nilpotence holds for $\widetilde X$ if and only if
      it holds for both $X$ and $Z$.
\end{enumerate}
\end{proposition}

\begin{proof}
The motivic blow-up formula gives
\[
 h(\widetilde X)\simeq h(X)\oplus
 \bigoplus_{j=1}^{c-1}h(Z)(j);
\]
see Manin \cite{Manin} or Fulton \cite[Section~6.7]{Fulton}.  Part (a)
follows from Lemma \ref{lem:direct-sum-bound} again.  For (b), the forward implication
uses the fact that $h(X)$ and, since $c\geq2$, $h(Z)(1)$ are direct summands
of $h(\widetilde X)$; the reverse implication follows from (a), since Tate
twists do not change base-change kernels in endomorphism rings.
\end{proof}

\begin{remark}
The qualitative equivalence in Proposition \ref{prop:blow-up-bound} is
\cite[Corollary~3.5]{Gille2026}; the explicit exponent in part (a) is the
quantitative refinement used here.
\end{remark}

\begin{corollary}
\label{cor:threefold-blow-up}
Let $X$ be a smooth projective threefold over a field of characteristic
zero and let $\widetilde X$ be the blow-up along a smooth centre of pure
codimension at least two.  If $a$ is a uniform integral Rost exponent for
$h(X)$, then $a+1$ is one for $h(\widetilde X)$.
\end{corollary}

\begin{proof}
The centre is a curve or a zero-dimensional scheme.  In the first case
base change on the endomorphism ring of its motive is injective by the Hochschild--Serre argument; in the second case its Tate-twisted
''defect'' motive is Artin--Tate.  Thus the whole ''defect'' motive in the blow-up
formula has zero base-change kernel and exponent $1$.  Apply Lemma
\ref{lem:direct-sum-bound} again to $h(X)$ and the ''defect'' motive.
\end{proof}

\begin{remark}
Consequently, a weak factorisation of length $s$ between smooth projective
threefolds transports a uniform exponent $a$ to one at most $a+s$.  This is
the uniform ideal-level quantitative form of
\cite[Theorem~2.4]{Gille2018}.
\end{remark}

\begin{corollary}\label{cor:birational-four}
Over a field of characteristic zero, uniform strong Rost nilpotence with
integral coefficients is a birational invariant of smooth projective
integral varieties of dimension at most four.  Consequently, every smooth
projective rational variety of dimension at most four has uniform strong
Rost nilpotence.
\end{corollary}

\begin{proof}
By weak factorisation, a birational map between smooth projective varieties
over a characteristic zero field is a sequence of blow-ups and blow-downs
with smooth centres \cite[Theorem~0.0.1]{Wlodarczyk}.  A nontrivial centre in
dimension $d\leq4$ has dimension at most $d-2\leq2$.  Each of its connected
components has uniform exponent $5$, so the whole centre has a finite
uniform exponent by Lemma \ref{lem:direct-sum-bound} and Corollary \ref{cor:low-dimensional-sums}.
Apply Proposition \ref{prop:blow-up-bound} at each step and use the fact that projective
space has zero base-change kernel.
\end{proof}

\begin{proposition}
\label{prop:generically-finite-morphism}
Let $f:Y\to X$ be a dominant generically finite morphism of degree $n$
between smooth projective integral varieties of the same dimension.  Let
$R$ be a coefficient ring in which $n$ is invertible.  If $s$ is a uniform
Rost exponent for $h(Y)$ with coefficients in $R$, then $s$ is a uniform
Rost exponent for $h(X)$ with coefficients in $R$.
\end{proposition}

\begin{proof}
The graph and transpose graph define morphisms $f_*:h(Y)\to h(X)$ and
$f^*:h(X)\to h(Y)$.  The projection formula and $f_*[Y]=n[X]$ give
\[
 f_*f^*=n\id_{h(X)}.
\]
After inverting $n$, the motive $h(X)$ is therefore a direct summand of
$h(Y)$.  The result follows from the corner algebra argument once again.
\end{proof}

\begin{corollary}\label{cor:generically-finite}
Let $d\leq4$ and let $Y\dashrightarrow X$ be a dominant generically finite
rational map of degree $n$ between smooth projective integral $d$-folds over
a field of characteristic zero.  If $Y$ has uniform strong Rost nilpotence
with integral coefficients, then $X$ has uniform strong Rost nilpotence with
coefficients in $\mathbb Z[1/n]$.  In particular, the conclusion holds for a
unirational $X$ parametrised by $\mathbb P^d\dashrightarrow X$ with degree
$n$.
\end{corollary}

\begin{proof}
Resolve the rational map to a generically finite morphism
$f:\widetilde Y\to X$.  By Corollary \ref{cor:birational-four}, $\widetilde Y$ has
uniform strong Rost nilpotence.  Flat localisation of the coefficient ring
preserves the kernel relation and its nilpotence.  Apply Proposition
\ref{prop:generically-finite-morphism} with $R=\mathbb Z[1/n]$.
\end{proof}

\section{An ideal-level refined unramified detector}

Throughout this section, Chow groups and Chow motives have integral
coefficients.  We recall only the further notation needed here.  For a
noetherian scheme $Y$, let $F_jY$ denote the pro-open subset consisting of points of
codimension at most $j$, and set
\[
 H^i(F_jY,A(r))=
 \varinjlim_{U\supset F_jY}H^i(U,A(r)).
\]
The $j$-th refined unramified cohomology group is
\begin{equation}\label{eq:refined-unramified-definition}
 H^i_{j,\mathrm{nr}}(Y,A(r))=
 \operatorname{im}\bigl(H^i(F_{j+1}Y,A(r))
 \longrightarrow H^i(F_jY,A(r))\bigr);
\end{equation}
see Schreieder \cite[Definition~5.1]{Schreieder}.  We use étale cohomology
with $A=\QlZl$.

Let $X$ be a smooth projective equidimensional scheme of dimension $d\geq3$
over a field of characteristic zero, and put $Y=X\times X$.  For every
prime $\ell$, Kok--Zhou construct a correspondence-equivariant exact sequence
\begin{multline}\label{eq:KZ-exact}
 H^{2d-2}_{d-3,\mathrm{nr}}(Y,\QlZl(d))
 \xrightarrow{\rho_\ell}
 \CH^d(Y)[\ell^\infty]
 \xrightarrow{\lambda_\ell}
 Q_\ell\\
 \longrightarrow
 H^{2d-1}_{d-2,\mathrm{nr}}(Y,\QlZl(d)),
\end{multline}
where
\[
 Q_\ell=
 \frac{H^{2d-1}(Y,\QlZl(d))}{M^{2d-1}_\ell(Y)}.
\]
Here
\[
 \begin{aligned}
 M^{2d-1}_\ell(Y)=\operatorname{im}\bigl(&
 H^{2d-1}_{\mathrm L}(Y,\mathbb Q_\ell(d))\\
 &\longrightarrow H^{2d-1}_{\mathrm L}(Y,\QlZl(d))\bigr),
 \end{aligned}
\]
where the subscript ''$\mathrm L$'' denotes Lichtenbaum
cohomology; see \cite[Definition~10.1]{MazzaVoevodskyWeibel} or \cite[Section~3]{RosenschonSawant}.  We use the canonical comparison 
$H^{2d-1}_{\mathrm L}(Y,\QlZl(d))\simeq
H^{2d-1}(Y,\QlZl(d))$ from \cite[Theorem~10.2]{MazzaVoevodskyWeibel} to regard this image as a subgroup of the numerator of
$Q_\ell$.  Exactness and equivariance are
\cite[Proposition~4.5 and Lemma~4.6]{KokZhou}.  A
correspondence $\gamma$ on $X$ acts on $Y$ through
$\gamma\times\Delta_X$; on $\CH^d(Y)=\End(h(X))$ this is right composition.

\begin{theorem}\label{thm:ideal-KZ}
Let $X$ be a smooth projective equidimensional scheme of dimension $d\geq3$
over a field $k$ of characteristic zero.  With integral coefficients, let
$L/k$ be a finite extension and put
\[
 J=\Ker\bigl(\End_k(h(X))\longrightarrow\End_L(h(X_L))\bigr).
\]
The following are equivalent.
\begin{enumerate}[label=\textup{(\roman*)}]
\item The ideal $J$ is nilpotent.
\item There is an integer $N\geq1$, independent of $\ell$, such that every
      element of $J^N$ acts trivially on
      \[
       H^{2d-2}_{d-3,\mathrm{nr}}
       (X\times X,\QlZl(d))
      \]
      for every prime $\ell$.
\end{enumerate}
If (ii) holds, then
\[
 J^{N+2d+1}=0.
\]
\end{theorem}

\begin{proof}
Since $\operatorname{char}k=0$, the finite extension $L/k$ is separable and
is contained in a finite Galois extension $L'/k$.  Put
$G=\operatorname{Gal}(L'/k)$, and let $A_\ell=\QlZl(d)$.  Every element of
$J$ vanishes after base change to $L'$.  The
Hochschild--Serre spectral sequence
\[
 E_2^{p,q}=H^p\bigl(G,H^q(Y_{L'},A_\ell)\bigr)
 \Longrightarrow H^{p+q}(Y,A_\ell)
\]
induces on $H_\ell=H^{2d-1}(Y,A_\ell)$ a finite filtration
\[
 H_\ell=F^0H_\ell\supseteq F^1H_\ell\supseteq\cdots
 \supseteq F^{2d}H_\ell=0.
\]
The standard construction of this spectral sequence is compatible with the
correspondence action; see \cite[Remark~3.2]{RosenschonSawant}.  If
$\gamma\in J$, then $\gamma_{L'}=0$ as a Chow correspondence; hence it acts
as zero on the $E_2$-page, and therefore as zero on the $E_\infty$-page and
on every associated graded piece.  It follows that
\[
 (F^pH_\ell)\gamma\subseteq F^{p+1}H_\ell,
 \quad H_\ell J^{2d}=0.
\]
The subgroup $M^{2d-1}_\ell(Y)$ is correspondence-stable, so
$Q_\ell J^{2d}=0$ as well.

Let $x\in\CH^d(Y)[\ell^\infty]$ and $a\in J^{2d}$.  Equivariance of
\eqref{eq:KZ-exact} gives $\lambda_\ell(xa)=0$, so exactness gives
$xa=\rho_\ell(z)$ for some element $z$ of the critical refined unramified
group.  If $b\in J^N$, hypothesis (ii) and equivariance give
\[
 xab=\rho_\ell(zb)=0.
\]
Thus
\begin{equation}\label{eq:torsion-annihilation}
 \CH^d(Y)[\ell^\infty]J^{N+2d}=0
 \qquad\text{for every }\ell.
\end{equation}

Put $n=[L':k]$.  Restriction to $L'$ followed by corestriction is
multiplication by $n$, and hence $nJ=0$.  In a product of $N+2d+1$ elements
of $J$, regard the first factor as a torsion cycle in $\CH^d(Y)$ and
decompose it into its
finitely many primary components.  Equation
\eqref{eq:torsion-annihilation} kills each component after multiplication by
the remaining $N+2d$ factors.  Therefore $J^{N+2d+1}=0$.

Finally, if $J^s=0$, then $J^s$ acts trivially on every functorial detector,
so (ii) holds with $N=s$.
\end{proof}

\begin{remark}
1.) For an individual element of $J$, Kok--Zhou prove that nilpotence is
equivalent to nilpotence of its action on the groups in
\cref{thm:ideal-KZ}, with a bound independent of $\ell$
\cite[Proposition~4.7]{KokZhou}.  The additional and interesting content of
\cref{thm:ideal-KZ} is the passage from individual elements to the whole
kernel ideal and the explicit exponent $N+2d+1$.

2.) Chow-weight homology provides a natural framework for producing, in higher
dimensions, the support-effectivity bounds required by Lemma 
\ref{lem:sharp-Rost-filtration}: its vanishing controls effectivity of
weight complexes, while smooth weight structures encode higher
birationality filtrations
\cite{BondarkoKumallagovCWH,BondarkoKumallagovBirational}.
What is still missing is a multiplicative refinement converting this
effectivity into uniform annihilation bounds for products of
correspondences.  Such a refinement would yield higher-dimensional
analogues of \cref{thm:ideal-RNZC-threefold}.
\end{remark}

\begingroup
\setlength{\emergencystretch}{2em}

\endgroup

\end{document}